\documentclass{amsart}

\usepackage[english]{my-shortcuts}

\usepackage{a4wide}
\usepackage{amsmath,amsfonts,amssymb,latexsym,amsthm}
\usepackage{comment} 
\usepackage{graphicx}
\usepackage{caption} 

\newtheorem{theorem}{Theorem}[section]
\newtheorem{corollary}[theorem]{Corollary}
\newtheorem{lemma}[theorem]{Lemma}

\begin{document}

\title{
Polynomial Moments with a weighted Zeta Square measure on the critical line}

\author{Sébastien Darses \and Erwan Hillion}

\address{IRL CRM--CNRS, Université de Montréal \and Aix-Marseille Universit\'e, cnrs, Centrale Marseille, I2M, Marseille, France} 
\email{sebastien.darses@univ-amu.fr}

\address{Aix-Marseille Universit\'e, cnrs, Centrale Marseille, I2M, Marseille, France}
\email{erwan.hillion@univ-amu.fr}

\date{}


\maketitle

\begin{abstract}
    We prove closed-form identities for the sequence of moments  \\ $\int t^{2n}|\Gamma(s)\zeta(s)|^2dt$ on the whole critical line $s=1/2+it$. They are finite sums involving binomial coefficients, Bernoulli numbers, Stirling numbers and $\pi$, especially featuring the numbers $\zeta(n)B_n/n$ unveiled by Bettin and Conrey in \cite{BC13a, BC13b}. 
    Their main power series identity \cite{BC13b}, together with \cite{DH21a}, allows for a short proof of an auxiliary result: the computation of the $k$-th derivatives at $1$ of the "exponential auto-correlation" function studied in \cite{DH21a}. We also provide an elementary and self-contained proof of this secondary result. The starting point of our work is a remarkable identity proven by Ramanujan in 1915. 
    The sequence of moments studied here, not to be confused with the moments of the Riemann zeta function, entirely characterizes $|\zeta|$ on the critical line. They arise in some generalizations of the Nyman-Beurling criterion, but might be of independent interest for 
    the numerous connections concerning the above mentioned numbers.
\end{abstract}


\section{Introduction}

\subsection{Main result}
The Gamma function and the Riemann zeta function are defined as usual as:
\bean
\Gamma(s) & = & \int_0^\infty e^{-x}x^{s-1}dx, \quad \mathfrak{R}(s)>0,\\
\zeta(s) & = &\sum_{n=1}^\infty \frac{1}{n^s}, \quad \mathfrak{R}(s)>1,
\eean
and by analytic continuation. 
We often write $\Gamma\zeta(s)=\Gamma(s)\zeta(s)$, for some complex number $s$, with real part $\mathfrak{R}(s)$. Throughout the paper, we use the following notations and convention:
\begin{itemize}
    \item $N,n,k,j$ are natural integers, possibly $0$ or $-1$; The convention $\sum_{n=2}^{0,1}=\sum_\varnothing=0$ holds;
    \item $i^2=-1$, and $\gamma=0,577215664\cdots$ is the Euler constant;
    \item $B_j$ denotes the $j$-th Bernoulli number (See Section \ref{bernoulli});
    \item $S(n,k)$ are the Stirling numbers of the second kind ($S(n,k)=0$ if $k>n$, see Section \ref{stirling});
    \item The notation for binomial coefficients is standard.
\end{itemize}

The main result of this paper is the following
\begin{theorem}\label{main}
For all $N\ge0$,
\bean
\frac{(-4)^N}{2\pi}\int_{-\infty}^\infty t^{2N}  \left|\Gamma\zeta\left(\frac{1}{2}+it\right)\right|^2 dt  & = & \log(2\pi)-\gamma -4N + \left(\frac{4^N}{2}-1\right) B_{2N}+ \sum_{j=2}^{2N} T_{2N,j}\frac{\zeta(j)B_{j}}{j},
\eean
where for all $j\ge2$,
\bean
T_{N,j} & = & (j-1)!\sum_{n=2}^{N} \binom{N}{n} 2^{n} \left[(-1)^n S(n+1,j) + (-1)^j S(n,j-1)\right].
\eean
\end{theorem}
Let us make four comments. 

(i) Notice that the $T_{2N,j}$'s are integers and the numbers $\zeta(j)B_j/j$ only involve $\pi$, $j$ and $B_{2j}$, since $2(2j)!\ \zeta(2j)=(-1)^{j+1}B_{2j}(2\pi)^{2j}$ and $B_{2j+1}=0$, $j\ge1$. We keep this elegant notation in reference to the connection with the fundamental works of Bettin and Conrey \cite{BC13a, BC13b} (see below for more details). \\

(ii) The quantities under study
\bean
M^{\Gamma\zeta}_{k} & = & \int_{-\infty}^\infty t^{k}  \left|\Gamma\zeta\left(\frac{1}{2}+it\right)\right|^2 dt, \quad k\ge0,
\eean
are moments related to a measure with a density where $|\zeta|^2$ appears. These moments are not to be confused with the moments of the Riemann zeta function 
\bean
\int_0^T\left|\zeta\left(\frac{1}{2}+it\right)\right|^{2k}dt, \quad T>0, \quad k\ge1,
\eean
which 
constitute a very important topic in analytic number theory, and has been the subject of huge research works; See recently e.g. \cite{CK19, HRS19, Naj20} and the fundamental references and connections therein.\\

(iii) The sequence $M^{\Gamma\zeta}$ entirely characterizes $|\zeta|$ on the critical line. Indeed, on one hand, $|\Gamma(s)|^2$ is entirely known due to Euler's reflection formula (where we set $s=1/2+it$ for real $t$):
\bean
|\Gamma(s)|^2 = \Gamma(s)\Gamma(1-s) =  \frac{\pi}{\sin(\pi s)} = \frac{2\pi}{e^{\pi t}+e^{-\pi t}}.
\eean
We could have written $1/\cosh (\pi t)$ within the measure instead of $\Gamma$,  but we keep this notation in reference to its basic relationship with Mellin-Plancherel isometry (see Lemma \ref{fourier-ram}). It is also known that the measure $dt/\cosh (\pi t)$ satisfies remarkable self-reciprocity properties regarding Fourier transform, see e.g. \cite{DKS08} and \cite[p.125]{God15}.

On the other hand, using $|\Gamma(s)|^2=O(e^{-\pi t})$ and the crude bound $\zeta(s)=O(t)$ as $t\to+\infty$ (see e.g. \cite{Ten22}), one has
\bean
\int_{-\infty}^\infty |t|^{k}  \left|\Gamma\zeta\left(s\right)\right|^2 dt
    = O\left(\int_{0}^\infty t^k e^{-t}dt\right) = O(k!)\ ,
\eean
so the Hamburger moment problem  $(-\infty,\infty)$ is determinate (\cite[Prop.1.5 p.88]{Sim98}), i.e. $|\Gamma\zeta(s)|^2dt$ is the only measure verifying Theorem \ref{main}. Notice that $M^{\Gamma\zeta}_{2k+1}=0$ since $t\mapsto |\Gamma\zeta(s)|$ is even.\\

(iv) The right hand side of the main formula is an alternate quantity, which is the $2N$-th derivative at $0$ of the following function $G$ unveiled by Ramanujan in 1915:
\begin{theorem}[\cite{Ram15}, Eq. (22) p .97]\label{rama}
For all real $v$,
\bea \label{ram}
\int_0^\infty \left|\Gamma\left(-\frac{1}{4}+i\frac{t}{4}\right)\right|^2 \Xi\left(\frac{t}{2}\right)^2 \frac{\cos(vt)}{1+t^2}dt
    & = & \pi\sqrt{\pi}\ G(v),
\eea
where $\Xi(t)=\xi(\frac{1}{2}+it)=\frac{1}{2}s(s-1) \pi^{-s/2}\Gamma(s/2)\zeta(s) $, and
\bea
G(v) & = & \int_0^\infty\left(\frac{1}{e^{xe^v}-1}-\frac{1}{xe^v}\right)\left(\frac{1}{e^{xe^{-v}}-1}-\frac{1}{xe^{-v}}\right)dx.
\eea
\end{theorem}
The left hand side of our main formula is basically the $2N$-th derivative at $0$ of the left hand side of (\ref{ram}) (See Section \ref{section-rama}).

The function $G$ is related by a simple change of variable to the "exponential auto--correlation" function $A$ introduced in \cite{DH21a}:
\bean
A(w) & = & \int_0^\infty \left(\frac{1}{xw}-\frac{1}{e^{xw}-1}\right)\left(\frac{1}{x}-\frac{1}{e^{x}-1}\right)dx, \qquad w>0.
\eean

\subsection{Secondary result}
We can now state our secondary result on which Theorem \ref{main} heavily relies. As usual, $\delta_{k,j}$ denotes the Kronecker symbol, and let $H_k$ be the harmonic number:
\bean
H_k & = & \sum_{j=1}^k \frac{1}{j}, \qquad k\ge 1,
\eean
with the convention $H_{-1}=H_0=0$. Set
\bea
C=\frac{\log(2\pi)-\gamma}{2}=0.6303307\cdots
\eea 
\begin{theorem}\label{second}
For all $k\ge0$, 
\bean
A^{(k)}(1)
     & = & (-1)^k k! \left( (1+\delta_{k,0})C -\frac{1}{2(k+1)} - \frac{H_{k-1}}{2}  + \sum_{j=2}^{k} {k \choose j-1} \frac{\zeta(j)B_j}{j} \right).
\eean
\end{theorem}
The last sum term of this equality is exactly the same term as in Theorem 1 -- Lemma 1 in \cite{BC13b}, since $A(x)$ is related to their "period" function $\psi$ up to some $1/x$ and $\log(x)$ factors, as shown in \cite{DH21a}.

These last facts then allow for a short proof of Theorem \ref{second}, see Section \ref{bettin-conrey}. We also provide a self-contained and elementary proof, whose various techniques might be of independent interest, showing e.g. how Bettin and Conrey's numbers arise from a combinatorial and real analysis setting.

\subsection{Previous works and motivation}

The so-called second moment of zeta $\int_0^T|\zeta(s)|^2dt$ appears in a variety of contexts and is well understood since Hardy and Littlewood in 1916, see \cite[Chap. VII]{Tit86}. Integrating on $(0,\infty)$ requires a weight, and we encounter the denomination "weighted moment" in the literature. For instance, asymptotic expansion of $\int_0^\infty|\zeta(s)|^2 e^{-\delta t}dt$, $\delta>0$, can be obtained, see e.g. \cite[Theorem 7.15 p.164]{Tit86}, and \cite{BC13a} for new remarkable formulas with convergent asymptotic series (See also \cite{BC13c}). Second moments are also used and studied for other Dirichlet series, see e.g. \cite{BISS18, ABBRS19} for various applications combining interesting tools.
Higher weighted moments of zeta, especially the fourth one, are also studied, see e.g. \cite[Chap. VII]{Tit86}, \cite{IM06}, \cite{BC13a}, and references therein.

The motivation for studying $M^{\Gamma\zeta}$ stems from some generalization of the Nyman-Beurling criterion (NB) for the Riemann hypothesis (RH). NB is an approximation problem of the indicator function of $(0,1)$ in $L^2(0,\infty)$ by linear combination of functions $t\mapsto\{\theta_k/t\}$ where $\theta_k\in(0,1]$, and $\{\cdot\}$ is the fractional part function.
B\'aez-Duarte \cite{BD03} showed one can take $\theta_k = 1/k$ in NB. We refer to \cite{dR07,DFMR13} where the authors consider generalizations of the zeta function.

RH then reads as a geometric problem studying a square distance in the Hilbert space $L^2(0,\infty)$; One then has to consider the scalar products
\bea
G_{k,j} & = & \int_0^{\infty} \left\{ \frac{1}{kt}\right\}\left\{ \frac{1}{jt}\right\} dt,\qquad 1\le k,j\le n,
\eea
where $n$ is the size of the corresponding Gram matrix. See \cite{BDBLS05} for a fine study of the corresponding auto-correlation function.
Vasyunin \cite{Vas95} proved that $G_{k,j}$ is a finite cotangent sum, which is connected to a variety of important objects and topics: Estermann function, reciprocity formulas, modular forms, Lewis--Zagier theory, see \cite{BCH85, LZ01, BC13a, MR16, ABB17}.




Replacing the $\theta_k$ in NB by random variables (r.v.)  $X_{k,n}$, $1\le k\le n$, produced new characterizations and structures, for instance using functions $t\mapsto \pE\{X_{k,n}/t\}$ ($\pE Z$ is the expectation of a r.v. $Z$). See \cite{DH21b} and an important generalization in \cite{ADH22}. Two main frameworks arise:
\begin{itemize}
    \item[(d)] {\em Dilation} --- Take e.g. exponential r.v. $X_k=X_1/k\sim \cal E(k)$; The corresponding square distance written with Mellin isometry involves Dirichlet polynomials, as in the original criterion, but a smoothing effect appears: see the auto-correlation function in \cite{DH21a};
    \item[(c)] {\em Concentration} --- Take e.g. Gamma distributed r.v. $X_{k,n}$ concentrated in $n$ around $1/k$; The square distance involves polynomials and the structure of zeta is then contained within the measure $|\Gamma\zeta(s)|^2dt$. The moments directly appear in the Gram matrix 
    (see \cite{ADH22}), a block Hankel matrix with a symbol having power-like singularities (see e.g. \cite{Kra07}).
\end{itemize}
Surprisingly, the auto-correlation function $A$ plays a role in both scalar products: through its value at rational numbers $A(k/j)$ in (d), or its derivatives $A^{(k)}(1)$ in (c).

\subsection{Outline}

In Section \ref{reminder}, we first gather useful information and properties on Bernoulli and Stirling numbers that we will use throughout the paper. Second, we relate the polynomial moments with the derivatives of the remarkable function unveiled by Ramanujan in 1915. Third, we set the tools to differentiate this function.

In Section \ref{sec:proofmain}, we prove Theorem \ref{main}, especially including the short proof of Theorem \ref{second} based on \cite{BC13b} and \cite{DH21a}.
Section \ref{sec:elementary} is devoted to the elementary proof of Theorem \ref{second}, based on several tools developed in Section \ref{reminder}. Finally Section~\ref{sec:numerical} exposes numerical results.

\medskip

\section{Reminders and Preliminary results}\label{reminder}

\subsection{Bernoulli numbers}\label{bernoulli}

The sequence of Bernoulli numbers $(B_n)_{n \ge 0}$ can be defined by its exponential generating function: 
\bea
\frac{t}{e^t-1} & = & \sum_{n=0}^\infty B_n \frac{t^n}{n!}, \qquad |t|<2\pi.
\eea
The generating function for the Bernoulli polynomials $B_n(\cdot)$ reads
\bea
\frac{te^{tx}}{e^t-1} & = & \sum_{n=0}^\infty B_n(x) \frac{t^n}{n!}, \qquad |t|<2\pi.
\eea
The $N$--th Bernoulli polynomial reads
\bean
B_N(x) & = & \sum_{n=0}^N \binom{N}{n} B_n x^{N-n}.
\eean
The first Bernoulli numbers are
\bean
B_0 = 1, \  \ B_1 = -\frac{1}{2}, \ \ B_2 = \frac{1}{6}, \ \ B_3 = 0, \  \ B_4 = -\frac{1}{30}, \ \ B_5=0, \ \ B_6=\frac{1}{42}, \ \cdots
\eean 

For $0\le j \le k$, we set:
\bea
K_{k,j} & = & \frac{1}{k} \binom{k}{j} B_{k-j} +\delta_{j,k-1}.
\eea

\subsection{Stirling numbers}\label{stirling}

\subsubsection{Change of basis}
We use the notation $(x)_n$ for the falling factorial polynomial:
\bean
(x)_n = x(x-1)\cdots(x-n+1), \quad n\ge1,
\eean
with $(x)_0=1$.
Both families $(1,x,x^2,\ldots,x^n)$ and $((x)_0,(x)_1,\ldots,(x)_n)$ are bases of the linear space of polynomial of degree at most $n$. The change of basis is ruled by Stirling numbers. 

The Stirling numbers of the first kind (signed), resp. second kind, is the family of coefficients $(s(n,k))_{n \ge 1, 1 \le k \le n}$, resp. $(S(n,k))_{n \ge 1, 0\le k \le n}$,  defined by 
\bean
(x)_n & = & \sum_{k=0}^n s(n,k) x^k, \\  
x^n & = & \sum_{k=0}^n S(n,k) (x)_k.
\eean
If $k \le 0$ or $k > n$ we set $s(n,k)=S(n,k)=0$. Stirling numbers satisfy the recurrence relations:  
\bean
s(n+1,k) & = & s(n,k-1)-n\ s(n,k) \\ 
S(n+1,k) & = & S(n,k-1)+k\ S(n,k).
\eean
The interplay between Stirling numbers of the first and second kind transcripts with matrices. Let $S_1 = (s(n,k))_{0 \le n,k \le N}$ and $S_2 = (S(n,k))_{0 \le n,k \le N}$. As an example, for $N=4$, 
\bean
S_2 = \begin{pmatrix} 1&0&0&0&0 \\ 0&1&0&0&0 \\ 0&1&1&0&0 \\ 0&1&3&1&0 \\ 0&1&7&6&1  \end{pmatrix},\quad 
S_1 = \begin{pmatrix} 1&0&0&0&0 \\ 0&1&0&0&0 \\ 0&-1&1&0&0 \\ 0&2&-3&1&0 \\ 0&-6&11&-6&1 \end{pmatrix}.
\eean
One then have the fundamental relation $S_2 = S_1^{-1}$. By computing each coefficient of $S_1S_2 = \rm{Id}_N$, one obtains that for any couple of indices $j \le k$: 
\bean
\sum_{p=j}^k S_2(k,p) s(p,j) & = & \delta_{j,k}.
\eean
The following identity, see \cite[(6.100)]{GKP94}, will be used in the elementary proof of Theorem \ref{second}:
\bea
\sum_{p=j}^k \frac{S(k,p) s(p,j)}{p} & = & K_{k,j}.
\eea
This formula is obtained writing a finite sum $\sum_k n^k$ using two times the change of basis by means of Stirling numbers, and then by identification using the Faulhaber formula.

\subsubsection{Operator and generating function}
Let $(u_k)_{k \ge 0}$ be a real sequence such that the power series 
\bean
F_u(x) & = & \sum_{p \ge 0} u_p x^p
\eean
has radius of convergence $r>0$.
Denote by $\cal E(u)$ the sequence defined by:
\bean
\cal E(u)_n & = & \sum_{k=0}^n S(n,k) (-1)^k k!\ u_{k},
\eean
and define the sequence $U$ by
\bean
U_n & = & \frac{(-1)^n}{n!} \cal E(u)_n.
\eean
The following lemma uses standard techniques of exponential generating functions, and will be very useful in the sequel.
\begin{lemma} \label{prop:petitesomme}
For all $t\in [0,\log(1+r)[$,
\begin{equation*}
F_U(t) = F_u(1-e^{-t}).
\end{equation*}
\end{lemma}

\begin{proof}
Let $t\in [0,\log(1+r)[$, i.e. $0\le e^t-1<r$ and $0\le 1-e^{-t}<r$. We have, since $S(n,k)\ge0$,
\bean
\sum_{n \ge 0}\sum_{k=0}^n \frac{1}{n!}S(n,k) k!\ |u_{k}| t^n= \sum_{k \ge 0} |u_k| \left(k! \sum_{n \ge k} S(n,k) \frac{t^n}{n!} \right)=\sum_{k=0}^\infty |u_k| (e^{t}-1)^k <\infty.
\eean
Therefore, by Fubini Theorem,
\bean
\sum_{n \ge 0} U_n t^n &=& \sum_{k \ge 0} (-1)^k u_k \left(k! \sum_{n \ge k} S(n,k) (-1)^n \frac{t^n}{n!} \right) \\
&=& \sum_{k\ge0} (-1)^k u_k (e^{-t}-1)^k \\
&=& \sum_{k\ge0} u_k (1-e^{-t})^k,
\eean
as desired.
\end{proof}

\subsection{Ramanujan's identity and the polynomial moments}
\label{section-rama}

In \cite{Ram15}, Ramanujan obtained remarkable identities concerning the $\Xi$ and $\xi$ functions, see \cite{D11,Kim16} for many investigations.

Fourier type transform involving the $\Xi$-function (and not its square) is a vast subject, see e.g. \cite[2.6 p.35]{Tit86}, and we refer to \cite{DRRZ15} and references therein for recent applications. 

The identity (\ref{ram}), which is Eq. (22) p.97 in \cite{Ram15}, is today interpreted as a consequence of Mellin-Plancherel isometry (See \cite{Kim16} for generalizations and the full treatment of Ramanujan's identities contained in \cite{Ram15}). 
For the sake of completeness, we recall the Mellin-Plancherel isometry argument to obtain the following lemma, which is only a reformulation of Ramanujan's identity (\ref{ram}) in terms of $\zeta$ and $\Gamma$:

\begin{lemma}[Reformulation of Eq.(22) in \cite{Ram15}] \label{fourier-ram}
For all real $v$,
\bean
\frac{1}{2\pi}\int_{-\infty}^\infty \cos(2vt)  |\Gamma(s)\zeta(s)|^2 dt = G(v) = \int_0^\infty\left(\frac{1}{e^{xe^v}-1}-\frac{1}{xe^v}\right)\left(\frac{1}{e^{xe^{-v}}-1}-\frac{1}{xe^{-v}}\right)dx.
\eean
\end{lemma}

\begin{proof}
Let us recall the fundamental formula
\bea \label{mellin-gamma-zeta}
\int_0^\infty \left(\frac{1}{e^{x}-1}-\frac{1}{x}\right)x^{w-1}dx & = & \Gamma(w)\zeta(w), \quad 0<\mathfrak{R}(w)<1.
\eea
Set for real $v$,
\bea
f_v(x) & = & \frac{1}{e^{xe^v}-1}-\frac{1}{xe^v}, \qquad x>0.
\eea
Therefore, we have in particular for $s=\frac{1}{2}+it$:
\bean
\widehat{f_v}(s) = \int_0^\infty \left(\frac{1}{e^{xe^v}-1}-\frac{1}{xe^v}\right)x^{s-1}dx 
     \ =\  e^{-vs} \widehat{f_0}(s) 
     \ = \ e^{-vs} \Gamma(s)\zeta(s).
\eean
Then, by Mellin-Plancherel isometry:
\bean
G(v)=\int_0^\infty f_v(x)f_{-v}(x)dx & = & \frac{1}{2\pi}\int_{-\infty}^\infty \widehat{f_v}(s)\widehat{f_{-v}}(\b s)dt \\
     & = & \frac{1}{2\pi}\int_{-\infty}^\infty e^{-vs+v\b s}\  |\Gamma(s)\zeta(s)|^2 dt \\
     & = & \frac{1}{2\pi}\int_{-\infty}^\infty \cos(2vt)  |\Gamma(s)\zeta(s)|^2 dt,
\eean
since $-vs+v\b s=-v(\frac{1}{2}+it)+v(\frac{1}{2}-it)=-2ivt$, and $t\mapsto \sin(2vt)|\Gamma(s)\zeta(s)|^2$ is odd.
\end{proof}
We can now differentiate under the integral sign to obtain
\bean
G^{(2N)}(v) & = & \frac{(-1)^N}{2\pi}2^{2N}\int_{-\infty}^\infty \cos(2vt)t^{2N}  |\Gamma(s)\zeta(s)|^2 dt,
\eean
and then taking $v=0$, we get the immediate
\begin{corollary} \label{moment-ram}
For all $N\ge0$,
\bean
\int_{-\infty}^\infty t^{2N}  \left|\Gamma\zeta\left(\frac{1}{2}+it\right)\right|^2 dt  & = & 2\pi(-1)^N 2^{-2N}G^{(2N)}(0).
\eean
\end{corollary}

The change of variable $y=xe^{-v}$ in $G$ directly yields:
\begin{lemma}\label{AG}
For all real $v$,
\bea
G(v) & = & e^v A(e^{2v}).
\eea
\end{lemma}
The main point of the paper is to differentiate $G$. To do so for the composed function $v\mapsto A(e^{2v})$, we will need the two lemmas of the following subsection.
\medskip

\subsection{Differentiating functions composed with the exponential}
Let us consider the differential operator $\cal D:C^\infty(\R)\to C^\infty(\R)$:
\bean
(\cal D \varphi)(x) & = & x\varphi'(x),
\eean
which is a basic example of a Cauchy-Euler operator.
We set $\cal D^{n+1}=\cal D\circ \cal D^n$ for all $n\ge0$ with $\cal D^{0}={\rm Id}$. The following lemma is the first question of Exercise 13 p. 300 in \cite{GKP94}:
\begin{lemma}\label{Dn}
For all $\varphi\in C^\infty$ and $n\ge0$, we have
\bean
\cal D^{n} \varphi(x) & = & \sum_{k=0}^n S(n,k) x^k \varphi^{(k)}(x).
\eean
\end{lemma}

\begin{proof}
We check that $\cal D^0\varphi(x)=S(0,0)\varphi^{(0)}(x)=\varphi(x)$ and $\cal D^1\varphi(x)=x\varphi'(x)=S(1,1)x^1\varphi^{(1)}(x)$. Assume the equality holds for some $n\ge1$. Since $\cal D$ is linear, we have
\begin{eqnarray*}
\cal D^{n+1} \varphi(x) & = & \sum_{k=0}^n S(n,k) x(kx^{k-1} \varphi^{(k)}(x)+x^k \varphi^{(k+1)}(x)) \\
    & = & S(n,1) x \varphi'(x)+ \sum_{k=2}^n [k S(n,k)+ S(n,k-1)]x^k \varphi^{(k)}(x) + S(n,n) x^{n+1} \varphi^{(n+1)}(x) \\
    & = & S(n,1) x \varphi'(x)+ \sum_{k=2}^n S(n+1,k) x^k \varphi^{(k)}(x) + S(n,n) x^{n+1} \varphi^{(n+1)}(x) \\
    & = & \sum_{k=0}^{n+1} S(n+1,k) x^k \varphi^{(k)}(x),
\end{eqnarray*}
where we used $S(n,n)=1=S(n+1,n+1)$, $S(n,0)=0=S(n+1,0)$ (since $n\ge1$) and $S(n,1)=1=S(n+1,1)$.
\end{proof}

\begin{lemma}\label{compo-exp}
Let $\phi(x)=\varphi(e^x)$ for any real number $x$. Then, for all $n\ge 1$,
\bean
\phi^{(n)}(x) & = & (\cal D^n \varphi)(e^x) \\
        & = & \sum_{k=0}^n S(n,k) e^{kx}\varphi^{(k)}(e^x).
\eean
\end{lemma}
\begin{proof}
We have $\phi'(x)=e^x \varphi'(e^x)=(\cal D \varphi)(e^x)$. By induction, assume that for some $n\ge1$, $\phi^{(n)}(x) = (\cal D^n \varphi)(e^x)$. Then
\bean
\phi^{(n+1)}(x) =  e^x(\cal D^n \varphi)'(e^x)  =  (\cal D (\cal D^n \varphi))(e^x),
\eean
and the conclusion follows.
\end{proof}

\bigskip

\section{Proof of the main result Theorem \ref{main}} \label{sec:proofmain}

In view of Corollary \ref{moment-ram} and Lemma \ref{AG}, we need to compute the derivatives of $A$ (Theorem \ref{second}), and then those of $G$.

\subsection{Computation of the derivatives of $A$} \label{bettin-conrey}
We give a short proof of Theorem \ref{second} based on Bettin  and Conrey's power series identity for the function $\psi$ below.

Following \cite{BC13a,BC13b}, consider, for $\mathfrak{Im}(z)>0$, 
\bean
E_1(z) & = & 1-4\sum_{n\ge1}d(n)e^{2\pi i n z} \\
\psi(z) & = & E_1(z)-1/z\ E_1(-1/z),
\eean
and their analytic continuation to $\C\setminus (-\infty,0]$ ($d(n)$ is the number of divisors of $n$).

For $|z|<1$, Bettin and Conrey \cite[Lemma 1, Part two]{BC13b} prove
\bea
\psi(1+z) & = & \frac{2i}{\pi}\sum_{k\ge0} \psi_k z^k,
\eea
with
\bea
\psi_k & = & \frac{(-1)^k}{k+1} + 2 \sum_{j=1}^{k-1}(-1)^{k-j}{k \choose j} \frac{\zeta(j+1)B_{j+1}}{j+1}, \quad k\ge2,
\eea
and $\psi_0=1$, $\psi_1=-1/2$.

Moreover, comparing the theorems \cite[Theorem 1]{BC13b} and \cite[Theorem 1]{DH21a}, we happily found a connection between $\psi$ and $A$ (See \cite{DH21a} and Lemma \ref{reform-BC} below). This turns out to be an almost direct consequence of \cite[Lemma 1, Part one, p.5717]{BC13b}. It is not really obvious that the inverse Mellin transform of their function $Q(s)=\zeta(s)\zeta(1-s)/\sin(\pi s)$ is $A$, up to a multiplicative constant, while it is straightforward to see that the Mellin transform of $A$ in the critical strip is basically $Q$.
\begin{lemma}[Reformulation of Lemma 1, Part one, in \cite{BC13b}] \label{reform-BC}
For all $x>0$
\bean
    A(x) & = & \frac{i\pi}{4} \psi(x) + r(x),
\eean
where
\bean
r(x) & = & C \left(1+\frac{1}{x}\right) - \frac{1}{2}\left(1-\frac{1}{x}\right) \log(x).
\eean
\end{lemma}

We can now differentiate $A$. On one hand,
\bean
r^{(k)}(x) & = & \frac{(-1)^k k!}{x^{k+1}}C-\frac{1}{2}\sum_{j=0}^k{k \choose j}(1-1/x)^{(j)}\log(x)^{(k-j)} \\
    & = & \frac{(-1)^k k!}{x^{k+1}}C + \frac{1}{2}\sum_{j=1}^{k-1}{k \choose j}\frac{(-1)^j j!}{x^{j+1}}\frac{(-1)^{k-j-1} (k-j-1)!}{x^{k-j}} \\
    &   & \qquad \qquad -\frac{1}{2}(1-1/x)\frac{(-1)^{k-1} (k-1)!}{x^{k}} + \frac{1}{2} \frac{(-1)^k k!}{x^{k+1}}\log(x).
\eean
Therefore
\bean
r^{(k)}(1) & = & (-1)^k k!C - \frac{(-1)^k k!}{2}\sum_{j=1}^{k-1}\frac{1}{k-j}\\
    & = & (-1)^k k! \left(C-\frac{H_{k-1}}{2}\right),
\eean
with the convention $H_0=0$.
On the other hand, with Taylor formula,
\bean
\frac{i\pi}{4}\psi^{(k)}(1) & = & \frac{i\pi}{4}\cdot \frac{2i}{\pi} \psi_k \ k!\\
        & = & -\frac{1}{2}\frac{(-1)^k k!}{k+1}- (-1)^k k!\sum_{j=2}^{k}(-1)^{j-1}{k \choose j-1} \frac{\zeta(j)B_{j}}{j},
\eean
which yields to the desired expression in Theorem \ref{second}, noting that $(-1)^jB_j=B_j$ for $j\ge2$.

\bigskip

\subsection{Computation of the derivatives of $G$}
Let $u$ be a real sequence, recall the definition of $\cal E(u)$ and define the sequence $\cal L(u)$:
\bean
\cal E(u)_n & = & \sum_{k=0}^n S(n,k) (-1)^k k!\ u_{k}\\
\cal L(u)_N & = & \sum_{n=0}^N {N\choose n} 2^{n} u_{n}.
\eean
The $\cal E$ is used to suggest that it is related to differentiating a function composed by an exponential. 
We use the notation $\cal L$ to refer to the Leibniz rule to differentiate a product.

\begin{lemma}
For all $N\ge 1$,
\bean
G^{(N)}(0) & = & \left(\cal L\circ \cal E\left(c-\frac{1}{2}(\iota+\eta)+\beta\right)\right)_N,
\eean
where we define for all $k\ge0$,
\bean
c_k & = & C(1+\delta_{k,0}) \\
\iota_k & = & \frac{1}{k+1} \\
\eta_k  & = & H_{k-1} \\
\beta_k & = & \sum_{j=2}^{k} {k \choose j-1} \frac{B_{j}\zeta(j)}{j}.
\eean
\end{lemma}

\begin{proof}
Recall that 
\bean
G(v) & = & e^v A(e^{2v}).
\eean
Using the Leibniz rule, Lemma \ref{compo-exp} together with the composition $x\mapsto 2x$, we obtain
\bean
G^{(N)}(x) & = & \sum_{n=0}^N {N\choose n} 2^{n}\sum_{k=0}^n S(n,k) e^{2kx}A^{(k)}(e^{2x}),
\eean
and then
\bean
G^{(N)}(0) & = & \sum_{n=0}^N {N\choose n} 2^{n}\sum_{k=0}^n S(n,k) A^{(k)}(1).
\eean
Writing
\bean
A^{(k)}(1) & = & (-1)^k k! \left( c_k-\frac{1}{2}(\iota_k+\eta_k)+\beta_k \right),
\eean
and using the notations, yields the desired expression.
\end{proof}

In view of Corollary \ref{moment-ram} and the
linearity of the operator $\cal L\circ \cal E$, what remains to establish Theorem 1.1 is to
compute each of the four quantities $\cal L\circ \cal E(c)$, $\cal L\circ \cal E(\iota)$, $\cal L\circ \cal E(\eta)$ and $\cal L\circ \cal E(\beta)$. This task is accomplished in the next four lemmas.

\begin{lemma}\label{}
We have 
\bean
\cal L\circ\cal E(c)_N = C\sum_{n=0}^N {N\choose n} 2^{n} \sum_{k=0}^n S(n,k) (-1)^k k!(1+\delta_{k,0})  = (1+(-1)^N)C.
\eean
\end{lemma}

\begin{proof}
First, we have due to the definition of $S(n,k)$:
\bean
\sum_{k=0}^n S(n,k)(-1)^k k!  = \sum_{k=0}^n S(n,k) (-1)_k = (-1)^n,
\eean
and
\bean
\sum_{n=0}^N {N\choose n} 2^{n} (-1)^n & = & (-1)^N.
\eean
Second, notice that
\bean
\sum_{n=0}^N {N\choose n} 2^{n} \sum_{k=0}^n S(n,k) (-1)^k k!\ \delta_{k,0} = \sum_{n=0}^N {N\choose n} 2^{n} S(n,0) = {N\choose 0} S(0,0)=1.
\eean
\end{proof}

\begin{lemma}
We have
\bean
\cal E(\iota)_n & = & B_n  \\
\cal L\circ\cal E(\iota)_N & = & (2-2^N) B_N.
\eean
\end{lemma}

\begin{proof}
First:
\bean
F_\iota(x) = \sum_{k=0}^\infty \frac{x^k}{k+1}
     =  - \frac{\ln(1-x)}{x}, \quad 0<x<1.
\eean
Therefore, noting that $B_1+1=(-1)^1B_1$ and $B_n=(-1)^n B_n$ for $n\ge2$,
\bean
F_\iota(1-e^{-t}) = \frac{t}{1-e^{-t}}
            = \frac{t}{e^{t}-1} +t = \sum_{n=0}^\infty (-1)^n\frac{B_n}{n!} t^n.
\eean
But $F_\iota(1-e^{-t})=F_I(t)$ with $I_n=\frac{(-1)^n}{n!}\cal E(\iota)_n$. Therefore, by identification,
\bean
\cal E(\iota)_n & = & B_n.
\eean
Second:
Recall that the $N^{th}$ Bernoulli polynomial is
\bean
B_N(x) & = & \sum_{n=0}^N \binom{N}{n} B_n x^{N-n},
\eean
and $B_N(0)= B_N$. Thus
\bean
B_N\left(\frac{1}{2}\right) & = & \frac{1}{2^N}\sum_{n=0}^N \binom{N}{n} B_n 2^{n} \\
     & = & \frac{1}{2^N}\cal L\circ\cal E(\iota)_N.
\eean
But it is known that (see e.g. Corollary 9.1.5. p.5-6 in \cite{Coh07})
\bea
B_N\left(\frac{1}{2}\right) & = & \left(\frac{1}{2^{N-1}} - 1 \right) B_N.
\eea
Hence
\bean
\cal L\circ\cal E(\iota)_N & = & (2-2^N) B_N.
\eean
\end{proof}

\begin{lemma}
We have
\bean
\cal E(\eta)_n & = &  (-1)^n n + \delta_{n,1}\\
\cal L\circ\cal E(\eta)_N & = & 2N (1+(-1)^N).
\eean
\end{lemma}

\begin{proof}
First, we compute $\cal E(\iota)_n$. Since $H_{-1}=H_0=0$,
\bean
F_\eta(x) = \sum_{k\ge0} H_{k-1} x^k= \sum_{k\ge2} H_{k-1} x^k=\sum_{k\ge1} H_{k} x^{k+1}, \quad 0<x<1.
\eean
Thus, writing $H_k$ and using the Fubini--Tonelli theorem,
\bean
F_\eta(x)  = x\sum_{k=1}^\infty \sum_{1\le p\le k} \frac{x^k}{p} = x\sum_{p\ge1} \frac{1}{p} \sum_{k \ge p} x^k = \frac{x}{1-x} \sum_{p\ge1} \frac{x^p}{p} = - \frac{x\ln(1-x)}{1-x}.
\eean
Then
\bean
F_\eta(1-e^{-t}) = -\frac{(1-e^{-t})(-t)}{e^{-t}} = t(e^t-1)
            = \sum_{n\ge1} \frac{t^{n+1}}{n!} = \sum_{n\ge2} \frac{t^{n}}{(n-1)!}
\eean
But $F_\eta(1-e^{-t})=F_{\cal H}(t)$ with $\cal H_n=\frac{(-1)^n}{n!}\cal E(\eta)_n$. Therefore $\cal H_0=\cal H_1=0$ and 
\bean
\frac{(-1)^n}{n!}\cal E(\eta)_n & = & \frac{1}{(n-1)!}, \quad n\ge2.
\eean
Hence 
\bean
\cal E(\eta)_n & = & (-1)^n n, \quad n\ge2,
\eean
and $\cal E(\eta)_0=\cal E(\eta)_1=0$, which we can write for all $n\ge0$,
\bean
\cal E(\eta)_n & = & (-1)^n n + \delta_{n,1}.
\eean

Second:
\bean
\cal L\circ\cal E(\eta)_N & = & \sum_{n=0}^N \binom{N}{n} 2^{n} \left((-1)^n n + \delta_{n,1}\right) \\
    & = & 2N + \sum_{n=0}^N \binom{N}{n} (-2)^{n} n.
\eean
But 
\bean
\sum_{n=0}^N \binom{N}{n} (-2)^{n} n & = & \sum_{n=1}^N n\binom{N}{n} (-2)^{n} \\
    & = & N\sum_{n=1}^N \binom{N-1}{n-1} (-2)^{n} \\
    & = & N\sum_{n=0}^{N-1} \binom{N-1}{n} (-2)^{n+1} \\
    & = & -2N(-2+1)^{N-1} =2(-1)^N N.
\eean
Finally,
\bean
\cal L\circ\cal E(\eta)_N & = & 2N (1+(-1)^N).
\eean
\end{proof}

\begin{lemma}
We have
\bean
\cal L\circ\cal E(\beta)_N & = & \sum_{j=2}^N T_{N,j}\frac{\zeta(j)B_{j}}{j},
\eean
where 
\bean
T_{N,j} & = & (j-1)!\sum_{n=2}^N \binom{N}{n} 2^{n} \left[(-1)^n S(n+1,j) + (-1)^j S(n,j-1)\right].
\eean
\end{lemma}

\begin{proof}
Let us expand and rewrite
\bean
\cal L\circ\cal E(\beta)_N & = & \sum_{n=0}^N \binom{N}{n} 2^{n} \sum_{k=0}^n  S(n,k)(-1)^k k! \sum_{j=2}^{k} \binom{k}{j-1}  \frac{\zeta(j)B_{j}}{j}\\
    & = & \sum_{j=2}^N \sum_{n=2}^N \binom{N}{n} 2^{n} W_{n,j}\frac{\zeta(j)B_{j}}{j},
\eean
where $W_{n,j} = 0$ if $n < j$, and for $n \ge j$,
\bean
W_{n,j} & = & \sum_{k=j}^n S(n,k)(-1)^k k! {k \choose j-1}\\
    & = &\sum_{k=j-1}^n S(n,k)(-1)^k k! {k \choose j-1} - (-1)^{j-1}(j-1)! S(n,j-1).
\eean
Fix $j\ge2$. Then
\bean
\sum_{k=j-1}^n S(n,k)(-1)^k k! {k \choose j-1} & = & \cal E(u)_n,
\eean
where $u_k={k \choose j-1}\1_{k\ge j-1}$. We have 
\bean
F_u(x)=\sum_{k=0}^\infty u_k x^k & = &  \sum_{k=j-1}^\infty {k \choose j-1} x^{k}  \\
    & = & x^{j-1}  \sum_{k=j-1}^\infty {k \choose j-1} x^{k-(j-1)}  \\
    & = & \frac{x^{j-1}}{(1-x)^{j}}.
\eean
Thus
\bean
F_u(1-e^{-t}) & = & e^{t j}(1-e^{-t})^{j-1} \\
    & = &e^t (e^t-1)^{j-1}\\
    & = & (e^t-1)^{j}+(e^t-1)^{j-1}.
\eean
Since
\bean
(e^t-1)^{j} & = & j ! \sum_{n \ge j} \frac{S(n,j)}{n!}t^n \\ (e^t-1)^{j-1} & = & (j-1) ! \sum_{n \ge j-1} \frac{S(n,j-1)}{n !}t^n,
\eean
we deduce that for $n \ge j$ : 
\bean
U_{n,j} (-1)^n = (j-1)! \left[ j S(n,j) + S(n,j-1) \right] = (j-1)! S(n+1,j).
\eean
Hence
\bean
W_{n,j} = (j-1)! \left[(-1)^n S(n+1,j) + (-1)^j S(n,j-1) \right],
\eean
as desired.
\end{proof}

\bigskip

\section{Elementary Proof of Theorem \ref{second}} \label{sec:elementary}

\subsection{Decomposition of $A^{(k)}(1)$}

Set
\bea
h(x) & = & \frac{1}{e^x-1}, \qquad x>0.
\eea
We can differentiate $A$ under the integral sign:
\bea
A^{(k)}(1) & = & \int_0^\infty \left(\frac{(-1)^k k!}{x}-x^k h^{(k)}(x)\right)\left(\frac{1}{x}-\frac{1}{e^{x}-1}\right)dx,
\eea
by grouping the divergent terms as $x\to0$ inside the integral of $A^{(k)}(v)$ and using e.g. Lemma \ref{e-diff-h} with a uniform bound for $v\in [1-\eta,1+\eta]$ (small $\eta>0$). We can then develop and obtain:

\begin{lemma}
For all $k\ge1$,
\bean
A^{(k)}(1) & = & D_{3,k}^\e + D_{2,k}^\e -D_{1,k}^\e,
\eean
where
\bean
D_{1,k}^\e & = & \int_\e^\infty x^{k-1}h^{(k)}(x)dx \\
D_{2,k}^\e & = & \int_\e^\infty \frac{x^k h^{(k)}(x)}{e^{x}-1}dx \\
D_{3,k}^\e & = & (-1)^k k!\int_\e^\infty \frac{1}{x}\left(\frac{1}{x}-\frac{1}{e^{x}-1}\right)dx.
\eean
\end{lemma}


\subsection{The derivatives of $h$}

Set
\bea
\alpha_{k,p} & = & (-1)^p S(k,p) p!
\eea
The integrals $D_{1,k}^\e$ and $D_{2,k}^\e$ involve the derivatives of $h$. The paper \cite{GQ14} gives some expressions of these ones and interesting applications. For our purpose, we need a different formula, especially having an $e^x$ within the numerator:
\begin{lemma}
For all $k\ge1$,
\bea \label{h(k)}
h^{(k)}(x) & = & \sum_{p=1}^k \alpha_{k,p} \frac{e^{px}}{(e^x-1)^{p+1}}.
\eea
\end{lemma}

\begin{proof}
Setting $\varphi(x) = \frac{1}{x-1}$, we have for all $p \ge 1$,
\begin{equation*}
\varphi^{(p)}(x) = \frac{(-1)^p p!}{(x-1)^{p+1}}.
\end{equation*}
Noting that $h(x)=\varphi(e^x)$, we apply Lemma \ref{compo-exp} to deduce
\bean
h^{(k)}(x) & = & \sum_{p=1}^k S(k,p) e^{px} \frac{(-1)^p p!}{(e^x-1)^{p+1}},
\eean
as claimed.
\end{proof}

\begin{lemma}
For $k\ge0$ and $p\ge1$,
\bea
\alpha_{k,p-1}-\alpha_{k,p}  & = &  -\frac{\alpha_{k+1,p}}{p}.
\eea
\end{lemma}

\begin{proof}
This only requires the recursive definition of $S(k+1,p)$:
\bean
\alpha_{k,p-1}-\alpha_{k,p} & = & (-1)^{p-1} S(k,p-1) (p-1)!-(-1)^p S(k,p) p! \\
    & = & (-1)^{p+1} (p-1)! \left(S(k,p-1)+pS(k,p)\right) \\
    & = & (-1)^{p+1} (p-1)! S(k+1,p),
\eean
and the claim follows.
\end{proof}

\subsection{Asymptotic expansions involving $h$}
We gather here several asymptotic expansions, useful for the sequel.

\begin{lemma} \label{e-diff-h}
For all $a\ge 0$, as $\e\to 0$,
\bea
\epsilon^a \frac{e^{a\epsilon}}{(e^\epsilon-1)^{a+1}} & = & \frac{1}{\epsilon} + \frac{a-1}{2} +o(1).
\eea
\end{lemma}

\begin{proof}
We have
\bean
\epsilon^a \frac{e^{a\epsilon}}{(e^\epsilon-1)^{a+1}} 
     & = & \frac{1}{\epsilon}\frac{1+a\epsilon+o(\epsilon)}{(1+\epsilon/2+o(\epsilon))^{a+1}}.
\eean
But
\bean
(1+\epsilon/2+o(\epsilon))^{-(a+1)} & = & 1-\frac{a+1}{2}\epsilon +o(\epsilon),
\eean
and the conclusion follows.
\end{proof}

\begin{lemma}
As $\e\to 0$,
\bean
I(\epsilon)=\int_\e^\infty \frac{1}{t}\left(\frac{1}{t}-\frac{1}{e^{t}-1}\right)dt & = & -\frac{\log \epsilon}{2} + C + o(1),\\
\int_\e^\infty \frac{1}{t(e^{t}-1)}dt & = & \frac{1}{\epsilon}+\frac{\log \epsilon}{2} - C + o(1),\\
\int_\e^\infty \frac{1}{(e^{t}-1)^2}dt & = & \log(\epsilon) + \frac{1}{\epsilon} - \frac{1}{2} +o(1).
\eean
\end{lemma}

\begin{proof}
In \cite{Bal18}, Balazard identified the constant $C=\frac 12(\log 2\pi -\gamma)$, see \cite{DH21a} for his proof. Rewriting $C$, we obtain
\bean
C & = & 1-\int_0^1 \left (\frac{1}{t(e^t-1)}-\frac 1{t^2}+\frac 1{2t}\right) dt-\int_1^{\infty} \frac{dt}{t(e^t-1)}\\
    & = & \int_0^1 \left[\frac{1}{t}\left(\frac{1}{t}-\frac{1}{e^{t}-1}\right)-\frac 1{2t}\right] dt + \int_1^\infty \frac{1}{t}\left(\frac{1}{t}-\frac{1}{e^{t}-1}\right)dt \\
    & = & I(\epsilon) - \int_\epsilon^1 \frac{dt}{2t}+o(1) \\
    & = & I(\epsilon) +\frac{\log \epsilon}{2}+o(1).
\eean
The asymptotic development of $I(\epsilon)$ and the second quantity then follow. Moreover
\bean
\int_\e^\infty \frac{-1}{(e^{t}-1)^2}dt & = & \int_\e^\infty \frac{e^t-1}{(e^{t}-1)^2}dt - \int_\e^\infty \frac{e^t}{(e^{t}-1)^2}dt \\
    & = & [\log(1-e^{-t})]_\epsilon^\infty + \left[\frac{1}{e^{t}-1}\right]_\epsilon^\infty \\
    & = & -\log(1-e^{-\epsilon})- \frac{1}{e^{\epsilon}-1}\\
    & = & -\log(\epsilon) - \frac{1}{\epsilon} + \frac{1}{2} +o(1),
\eean
as desired.
\end{proof}

The following elementary quantities will be useful in the next section.
\begin{lemma} \label{sumq}
As $\e\to 0$,
\bea
\sum_{q\ge1}\frac{e^{-\epsilon q}}{q} & = & -\log \epsilon + o(1),
\eea
and for all $a\ge 1$, 
\bea
\epsilon^a\sum_{q\ge1}q^{a-1}e^{-\epsilon q} & = & (a-1)! + o(1) \\
\epsilon^a\sum_{q\ge1}q^{a}e^{-\epsilon q} & = & \frac{a!}{\epsilon} + o(1).
\eea
\end{lemma}

\begin{proof}
We have
\bean
\sum_{q\ge1}\frac{e^{-\epsilon q}}{q} = \int_\e^\infty \frac{dt}{e^{t}-1} =  [\log(1-e^{-t})]_\epsilon^\infty  
    =  -\log(1-e^{-\epsilon}),
\eean
and we obtain the first expansion.

For $a\ge2$, we have
\bean
\epsilon^a\sum_{q\ge1}q^{a-1}e^{-\epsilon q} & = & (-1)^{a-1}\epsilon^a h^{(a-1)}(\epsilon).
\eean
But 
\bean
\epsilon^a h^{(a-1)}(\epsilon) & = & 
\epsilon^a\sum_{p=1}^{a-1} S(a-1,p) e^{p\epsilon} \frac{(-1)^p p!}{(e^\epsilon-1)^{p+1}} \\
    & = & (-1)^{a-1}(a-1)!  \frac{\epsilon^a e^{(a-1)\epsilon}}{(e^\epsilon-1)^{a}} +o(1)\\
    & = & (-1)^{a-1}(a-1)! +o(1),
\eean
and the conclusion follows noticing that the identity also holds for $a=1$.

Finally, for $a\ge1$, 
\bean
\epsilon^a\sum_{q\ge1}q^{a}e^{-\epsilon q} & = & (-1)^{a}\epsilon^a h^{(a)}(\epsilon).
\eean
Notice that for $a\ge2$, as $\e\to0$,
\bean
\epsilon^a h^{(a)}(\epsilon) & = & 
\epsilon^a\sum_{p=1}^{a} S(a,p) e^{p\epsilon} \frac{(-1)^p p!}{(e^\epsilon-1)^{p+1}} \\
    & = & (-1)^{a-1}(a-1)! S(a,a-1) \frac{\epsilon^a e^{(a-1)\epsilon}}{(e^\epsilon-1)^{a}} + (-1)^{a}a!  \frac{\epsilon^a e^{a\epsilon}}{(e^\epsilon-1)^{a+1}} + o(1)\\
    & = & (-1)^{a-1}(a-1)! S(a,a-1) + (-1)^{a}a!\left(\frac{1}{\epsilon} + \frac{a-1}{2}\right)+o(1).
\eean
Since 
\bean
(a-1)! S(a,a-1) = (a-1)! {a \choose 2}  = \frac{a-1}{2} a!
\eean
we deduce
\bean
\epsilon^a\sum_{q\ge1}q^{a}e^{-\epsilon q} & = & \frac{a!}{\epsilon} +o(1).
\eean
and the conclusion follows noticing that the identity also holds for $a=1$.
\end{proof}

\subsection{Incomplete integrals related to $D_{1,k}^\e$ and $D_{2,k}^\e $}

The following quantities will be involved in the computation of $D_{1,k}^\e$ and $D_{2,k}^\e $. Set
\bean
J_1^\e(k,p) & = & \int_\e^\infty x^{k}\frac{e^{px}}{(e^x-1)^{p+1}}dx \\
J_2^\e(k,p) & = & \int_\e^\infty x^{k}\frac{e^{px}}{(e^x-1)^{p+2}}dx,
\eean
and
\bean
Z_\e(k,j) & = & \sum_{q=1}^\infty q^j \int_\e^\infty e^{-qy}y^k dy.
\eean

\begin{lemma}
For all $k,p\ge0$,
\bean
J_2^\e(k,p) & = &
J_1^\e(k,p+1)-J_1^\e(k,p),
\eean
\end{lemma}

\begin{proof}
A simple manipulation gives
\bean
J_2^\e(k,p) = \int_\e^\infty x^{k}\frac{e^{px}}{(e^x-1)^{p+2}}dx 
     & = & -\int_\e^\infty x^{k}\frac{e^{px}(e^x-1-e^x)}{(e^x-1)^{p+2}}dx \\
     & = & -\int_\e^\infty x^{k}\frac{e^{px}}{(e^x-1)^{p+1}}dx + \int_\e^\infty x^{k}\frac{e^{(p+1)x}}{(e^x-1)^{p+2}}dx,
\eean
as claimed.
\end{proof}

\begin{lemma}
For all $k,p\ge0$,
\bean
J_1^\e(k,p) =  \int_\e^\infty x^{k}\frac{e^{px}}{(e^x-1)^{p+1}}dx & = &
\frac{1}{p!} \sum_{j=0}^{p} (-1)^{p+j} s(p,j) Z_\e(k,j),
\eean
where 
\bean
Z_\e(k,j) & = & k! \sum_{a=0}^k \frac{\e ^a}{a!}\sum_{q=1}^\infty \frac{e^{-\e q}}{q^{k-j+1-a}}.
\eean
If $j<k$, i.e. $k-j+1\ge2$, then
\bean
Z_\e(k,j) & = & k! \ \zeta(k-j+1) +o(1).
\eean
\end{lemma}

\begin{proof}
Using changes of variables, we obtain
\bean
J_1^\e(k,p) & = & \int_\e^\infty x^{k}\frac{e^{px}}{(e^x-1)^{p+1}}dx \qquad (u=e^x,\ x=\log u)\\
    & = & \int_{e^\e}^\infty \frac{u^{p-1} \log^k u}{(u-1)^{p+1}}du \qquad (u=1/x)\\
    & = &  (-1)^{k}\int_0^{e^{-\e}} \frac{\log^k x}{(1-x)^{p+1}} dx.  
\eean
We have the expansion
\bean
\frac{1}{(1-x)^{p+1}} & = & \frac{1}{p!} \sum_{q=1}^\infty \frac{(q+p-1)!}{(q-1)!} x^{q-1}.
\eean
On the other hand,
\bean
\frac{(q+p-1)!}{(q-1)!} & = & q(q+1)\cdots (q+p-1) \\
    & = & (-1)^p (-q)(-q-1) \cdots (-q-(p-1))\\
    & = & (-1)^p \sum_{j=0}^{p} s(p,j) (-1)^{j} q^j,
\eean
where we recall that $s(p,j)$ is a (signed) Stirling number of the first kind. Using the Fubini--Tonelli theorem and the change of variables $y=-\log(x)$, we obtain
\bean
J_1^\e(k,p) & = & \frac{(-1)^{k}}{p!} \sum_{j=0}^{p} (-1)^{p+j} s(p,j) \sum_{q=1}^\infty q^j \int_0^{e^{-\e}} x^{q-1} \log^k x\  dx \\
     & = & \frac{1}{p!} \sum_{j=0}^{p} (-1)^{p+j} s(p,j) \sum_{q=1}^\infty q^j \int_\e^\infty e^{-qy}y^k dy.
\eean
Let us now set and compute, using the particular form of the incomplete Gamma function $\int_{\e q}^\infty e^{-x}x^k dx$ when $k$ is an integer:
\bean
Z_\e(k,j) & = & \sum_{q=1}^\infty q^j \int_\e^\infty e^{-qy}y^k dy \\
    & = & \sum_{q=1}^\infty q^j \frac{1}{q^{k+1}}\int_{\e q}^\infty e^{-x}x^k dx \\
     & = & k! \sum_{q=1}^\infty \frac{e^{-\e q}}{q^{k-j+1}} \sum_{a=0}^k \frac{(\e q)^a}{a!}.
\eean
Finally notice that if $j<k$, the dominated convergence theorem yields
\bean
Z_\e(k,j) & = & k! \ \zeta(k-j+1) +o(1),
\eean
and the conclusion follows.
\end{proof}


\begin{lemma}\label{Zkk}
The following expansions hold:
\bea
Z_\e(k,k) & = & -k!\log(\epsilon) + k!\ H_k + o(1), \qquad k\ge1; \\
Z_\e(k,k+1) & = & \frac{(k+1)!}{\epsilon}-\frac{k!}{2} + o(1) \qquad k\ge0. \label{Zkk+1}
\eea
\end{lemma}

\begin{proof}
First, we compute
\bean
Z_\e(k,k) & = &  k! \sum_{a=0}^k \frac{\epsilon^a}{a!}  \sum_{q\ge1}q^{a-1}e^{-\epsilon q}\\
    & = & k!\left(\sum_{q\ge1}q^{-1}e^{-\epsilon q}+\sum_{a=1}^k \frac{1}{a!} \epsilon^a \sum_{q\ge1}q^{a-1}e^{-\epsilon q}\right)\\
    & = & k!\left(-\log(\epsilon)+\sum_{a=1}^k \frac{1}{a} \right) + o(1).
\eean
Second,
\bean
Z_\e(k,k+1) & = &  k! \sum_{a=0}^k \frac{\epsilon^a}{a!}  \sum_{q\ge1}q^{a}e^{-\epsilon q}\\
    & = & k!\left(\sum_{q\ge1} e^{-\epsilon q}+\sum_{a=1}^k \frac{1}{a!} \epsilon^a \sum_{q\ge1}q^{a}e^{-\epsilon q}\right)\\
    & = & k!\left(\frac{1}{e^\epsilon-1}+\sum_{a=1}^k \frac{1}{a!}\frac{a!}{\epsilon} \right) + o(1)\\
    & = & k!\left(\frac{1}{\epsilon}-\frac{1}{2}+ \frac{k}{\epsilon} \right) + o(1),
\eean
as claimed. 
\end{proof}

\subsection{Asymptotic expansion of $D_{1,k}^\e$, $D_{2,k}^\e$, and conclusion}

\begin{lemma}
For all $k\ge1$,
\bean
D_{1,k}^\e   & = & (-1)^k \left(\frac{k!}{\epsilon}-\frac{(k-1)!}{2}\right) + o(1). 
\eean
\end{lemma}

\begin{proof}
We have for all $k\ge1$,
\bean
D_{1,k}^\e & = & \int_\e^\infty x^{k-1} h^{(k)}(x)dx \\
    & = & \int_\e^\infty x^{k-1}\sum_{p=1}^k \alpha_{k,p} \frac{e^{px}}{(e^x-1)^{p+1}} dx \\
    & = & \sum_{p=1}^k \alpha_{k,p} J_1^\e(k-1,p)\\
    & = & \sum_{p=1}^{k} (-1)^p S(k,p) \sum_{j=1}^{p} S_1(p,j) Z_\e(k-1,j)\\
    & = & \sum_{j=1}^{k} (-1)^{j}\sum_{p=j}^{k} S(k,p)S_1(p,j)\ Z_\e(k-1,j)\\
    & = & \sum_{j=1}^{k}(-1)^{j} \delta_{k,j} Z_\epsilon(k-1,j)\\
    & = & (-1)^{k}  Z_\epsilon(k-1,k),
\eean
and we use (\ref{Zkk+1}) replacing $k$ by $k-1$ to conclude.
\end{proof}

\begin{lemma}
We have
\bean
D_{2,k}^\e & = & \sum_{j=1}^{k+1}(-1)^{j+1} K_{k+1,j} Z_\epsilon(k,j).
\eean
\end{lemma}

\begin{proof}
We have
\bean
D_{2,k}^\e & = & \int_\e^\infty \frac{x^k h^{(k)}(x)}{e^{x}-1}dx \\
    & = & \int_\e^\infty \frac{x^k}{e^{x}-1}\sum_{p=1}^k \alpha_{k,p} \frac{e^{px}}{(e^x-1)^{p+1}} dx \\
    & = & \sum_{p=1}^k \alpha_{k,p} J_2^\e(k,p)\\
    & = & \sum_{p=1}^k \alpha_{k,p}\left(J_1^\e(k,p+1)-J_1^\e(k,p)\right)\\
    & = & \sum_{p=2}^{k+1} \alpha_{k,p-1}J_1^\e(k,p) - \sum_{p=1}^k \alpha_{k,p}J_1^\e(k,p) \\
    & = & \alpha_{k,k}J_1^\e(k,k+1)-\alpha_{k,1}J_1^\e(k,1) + \sum_{p=2}^{k} (\alpha_{k,p-1}-\alpha_{k,p})J_1^\e(k,p)\\
    & = & \alpha_{k,k}J_1^\e(k,k+1)-\alpha_{k,1}J_1^\e(k,1) - \sum_{p=2}^{k} \frac{\alpha_{k+1,p}}{p} J_1^\e(k,p) \\
    & = & - \sum_{p=1}^{k+1} \frac{\alpha_{k+1,p}}{p} J_1^\e(k,p), 
\eean
since $\d \frac{\alpha_{k+1,k+1}}{k+1}=(-1)^{k+1}k!=-\alpha_{k,k}$ and $\alpha_{k+1,1}=\alpha_{k,1}$.
Moreover
\bean
\sum_{p=1}^{k+1} \frac{\alpha_{k+1,p}}{p} J_1^\e(k,p)
    & = & \sum_{p=1}^{k+1} \frac{(-1)^p}{p}S(k+1,p) \sum_{j=1}^{p} (-1)^{p-j}S_1(p,j) Z_\e(k,j)\\
    & = & \sum_{j=1}^{k+1} \sum_{p=j}^{k+1}\frac{(-1)^j}{p}S(k+1,p)S_1(p,j)\ Z_\e(k,j)\\
    & = & \sum_{j=1}^{k+1}(-1)^{j} K_{k+1,j} Z_\epsilon(k,j),
\eean
which concludes the proof.
\end{proof}

For $k\ge2$, $D_{2,k}^\e$ contains convergent terms involving $\zeta$ and  a divergent term $R_{2,k}^\e$, which reads
\bean
R_{2,k}^\e & = & (-1)^{k+1} K_{k+1,k} Z_\epsilon(k,k)+ (-1)^{k} K_{k+1,k+1} Z_\epsilon(k,k+1) 
\eean
But 
\bean
K_{k+1,k} & = & \frac{1}{k+1} \binom{k+1}{k} B_{1} +1 \ =\ 1/2 \\
K_{k+1,k+1} & = &\frac{1}{k+1} \binom{k+1}{k+1} B_{0} \ =\  \frac{1}{k+1}.
\eean
Therefore, by Lemma \ref{Zkk},
\bean
R_{2,k}^\e & = & (-1)^{k+1} \frac{-k!\log(\epsilon) + k!\ H_k}{2} + (-1)^{k} \frac{1}{k+1}\left(\frac{(k+1)!}{\epsilon}-\frac{k!}{2}\right) + o(1). 
\eean

\begin{lemma}
We have
\bean
D_{3,k}^\e & = & (-1)^{k+1} k! \frac{\log \epsilon}{2} + (-1)^k k! C + o(1).
\eean
\end{lemma}

We can now complete the proof of Theorem \ref{second} using 
\bean
A^{(k)}(1) & = & D_{3,k}^\e + D_{2,k}^\e -D_{1,k}^\e.
\eean
Therefore for all $k\ge1$, with the convention $\sum_{j=1}^{0}=0$,
\bean
A^{(k)}(1) & = & (-1)^k k! \left(C - \frac{H_{k+1}}{2} +\frac{1}{2k} - \sum_{j=1}^{k-1}(-1)^{j+k} K_{k+1,j} \zeta(k-j+1) \right).
\eean
Using for all $j\le k-1$
\bea
K_{k+1,j} & = & \frac{1}{k+1} \binom{k+1}{j} B_{k+1-j},
\eea
symmetry and "pion" formula for binomial coefficients, and $(-1)^jB_j=B_j$, $j\ge2$, we can write:
\bean
\sum_{j=1}^{k-1}(-1)^{j+k} K_{k+1,j} \zeta(k-j+1)  & = & \sum_{j=1}^{k-1}(-1)^{j+k} \frac{1}{k+1} \binom{k+1}{j} B_{k+1-j} \zeta(k-j+1)  \\
     & = & -\sum_{j=2}^{k}(-1)^{j} \frac{1}{k+1} \binom{k+1}{j} B_{j} \zeta(j) \\
     & = & -\sum_{j=2}^{k} \frac{1}{j} \binom{k}{j-1} B_{j} \zeta(j) 
\eean
Hence, for all $k\ge1$, with the convention $H_0=0$,
\bean
A^{(k)}(1)
     & = & (-1)^k k! \left(C -\frac{1}{2(k+1)} - \frac{H_{k-1}}{2}  + \sum_{j=2}^{k} {k \choose j-1} \frac{\zeta(j)B_j}{j} \right).
\eean

To obtain the formula for $k=0$, let us finally compute $A(1)$.
\begin{lemma}\label{}
We have
\bean
A(1) & = & 2C -\frac{1}{2}.
\eean
\end{lemma}

\begin{proof}
We expand
\bean
A(1) & = & \int_\epsilon^\infty \left(\frac{1}{x}-\frac{1}{e^{x}-1}\right)\left(\frac{1}{x}-\frac{1}{e^{x}-1}\right)dx +o(1) \\
    & = & I(\epsilon) - \int_\epsilon^\infty \frac{x^{-1}dx}{e^x-1} +  \int_\epsilon^\infty \frac{dx}{(e^x-1)^2} +o(1)\\
    & = & -\frac{\log \epsilon}{2} + C  -\frac{1}{\epsilon} - \frac{\log \epsilon}{2} + C + \log(\epsilon) + \frac{1}{\epsilon} - \frac{1}{2} +o(1) \\
    & = & 2C -\frac{1}{2} +o(1),
\eean
and take the limit.
\end{proof}

\bigskip

\section{Numerical results}\label{sec:numerical}

\subsection{Expression of the first $A^{(k)}(1)$}

Recall that 
\bean
A(v) & = & \int_0^\infty \left(\frac{1}{xv}-\frac{1}{e^{xv}-1}\right)\left(\frac{1}{x}-\frac{1}{e^{x}-1}\right)dx, \qquad v>0,
\eean
and
\bean
C=\frac{\log(2\pi)-\gamma}{2}=0.6303307\cdots
\eean
We have $\d A(1)  =  2C -\frac{1}{2}$,
\bean
A'(1) = -\left(C-\frac{H_2}{2}+\frac{1}{2}\right) = -C+\frac{1}{4},
\eean
and
\bean
A''(1) & = & 2\left(C-\frac{H_3}{2}+\frac{1}{4} + {2 \choose 1} \frac{B_{2}\zeta(2)}{2}\right)\\
    & = & 2C-\frac{4}{3} + \frac{1}{3}\zeta(2).
\eean

To check these values, we directly differentiate $A$ and evaluate at $1$, for instance:
\bean
A'(1) & = & \int_0^\infty \left(-\frac{1}{x}+\frac{xe^x}{(e^{x}-1)^2}\right)\left(\frac{1}{x}-\frac{1}{e^{x}-1}\right)dx.
\eean

\subsection{Expression of the two first moments}


The formula for $N=1$ reads 
\bean
\frac{-2}{\pi}\int_{-\infty}^\infty t^{2}  \left|\Gamma\zeta\left(\frac{1}{2}+it\right)\right|^2 dt  & = & \log(2\pi)-\gamma - 4 + \left(\frac{4}{2}-1\right) B_{4}+ T_{2,2}\frac{\zeta(2)B_{2}}{2} \\
    & = & \log(2\pi)-\gamma - \frac{23}{6} + \frac{4}{3}\zeta(2),
\eean
since 
\bean
T_{2,2} & = &  \binom{2}{2} 2^2 \left[ S(3,2) + S(2,1)\right] \\
    & = & 4 (3+1)=16.
\eean


The formula for $N=2$ reads 
\bean
\frac{(-4)^2}{2\pi}\int_{-\infty}^\infty t^{4}  \left|\Gamma\zeta\left(\frac{1}{2}+it\right)\right|^2 dt  & = & \log(2\pi)-\gamma - 8 + \left(\frac{4^2}{2}-1\right) B_{4}+ \sum_{j=2}^{4} T_{4,j}\frac{\zeta(j)B_{j}}{j},
\eean
where 
\bean
T_{4,j} & = & (j-1)!\sum_{n=2}^{4} \binom{4}{n} 2^{n} \left[(-1)^n S(n+1,j) + (-1)^j S(n,j-1)\right].
\eean
Thus, we have
\bean
T_{4,2}  & = & 24(3+1) +32(-7+1) +16(15+1) \\
    & = & 160, \\
T_{4,4}  & = & 6(0+0+16(10+6)) \\
    & = & 1536.
    \eean
Hence
\bean
\frac{8}{\pi}\int_{-\infty}^\infty t^{4}  \left|\Gamma\zeta\left(\frac{1}{2}+it\right)\right|^2 dt  & = & \log(2\pi)-\gamma - 8 - 7\cdot \frac{1}{30} + 160\frac{\zeta(2)/6}{2} - 1536\frac{\zeta(4)/30}{4} \\
    & = & \log(2\pi)-\gamma - \frac{247}{30} + \frac{40}{3}\zeta(2) - \frac{64}{5}\zeta(4).
\eean

\bigskip 

\subsection{Numerical values}

The following table gives the first values of the $T(\ell,j)$:

\begin{equation*}
\begin{array}{| c | ccccccc |}
\hline
 \ell \backslash j & 2 & 3 & 4 & 5 & 6 & 7 & 8\\
\hline  
2 & 16 & & & & & & \\ 3 & 0 & -144 & & & & &\\ 4 & 160 & 0 & 1536 & & & &\\ 5 & 0 & -5280 & 0 & -19200 & & &\\ 6 & 1456 & 0 & 145920 & 0 & 276480 &  & \\ 7 & 0 & -147504 & 0 & -3897600 & 0 & -4515840 & \\
8 & 13120 &	0 &	9225216	& 0	& 105799680	& 0	& 82575360 \\
\hline 
\end{array}
\end{equation*}

\bigskip 

The last one gives the first moments $M^{\Gamma\zeta}_{k}$: 

\medskip 

\begin{center}
\begin{tabular}{|c|l|}
    \hline
    $k$ & $\qquad M^{\Gamma\zeta}_{k} \rule{0pt}{2.6ex} \rule{2.6ex}{0pt}$ \tabularnewline
    \hline
    0 \rule{0pt}{2.6ex}& $4.77937654\cdots$ \rule{0pt}{2.6ex} \\
    2 \rule{0pt}{2.6ex}& $0.59600176\cdots$ \rule{0pt}{2.6ex} \\
    4 \rule{0pt}{2.6ex}& $0.43434281\cdots $\rule{0pt}{2.6ex} \\
    6 \rule{0pt}{2.6ex}& $1.01613719\cdots $\rule{0pt}{2.6ex} \\
    8 \rule{0pt}{2.6ex}& $5.60532440\cdots$ \rule{0pt}{2.6ex} \\
    10 \rule{0pt}{2.6ex}& $57.6316873\cdots$ \rule{0pt}{2.6ex} \\
    12 \rule{0pt}{2.6ex}& $940.337401\cdots$ \rule{0pt}{2.6ex}
    \tabularnewline
    \hline
 \end{tabular}
\end{center}

\bigskip

\section*{Acknowledgement} 
The authors are very grateful to the anonymous referee for their report, comments and suggestions that improved the paper.

The authors thank Michel Balazard for noting the relation of Ramanujan's formulas with Mellin-Plancherel theorems and communicating the reference \cite{Kim16}. They thank Charles Bordenave for communicating \cite{Sim98}. They are also grateful to Ren\'e Adad, S\'er\'ena Pedon and G\'erald Tenenbaum for corrections, and to Ren\'e Adad and Joseph Najnudel for numerical experiments. 

The first author warmly thanks Francois Alouges, Michel Balazard, Paul Bourgade, Christophe Delaunay, Persi Diaconis, Bryce Kerr, Oleksiy Klurman, Joseph Najnudel, Olivier Ramar\'e and Kristian Seip for interesting references and insightful discussions.

\bigskip


\end{document}